\documentclass[11pt]{amsart}

\usepackage{amssymb, amsmath, amsthm, color, hyperref, url, fullpage}

\newtheorem{theorem}{Theorem}
\newtheorem{corollary}[theorem]{Corollary}
\newtheorem{lemma}[theorem]{Lemma}
\newtheorem{prop}[theorem]{Proposition}

\theoremstyle{definition}
\newtheorem{defn}[theorem]{Definition}
\newtheorem{example}{Example}

  \let\l\lambda   
      
\let\GL\Lambda
\def\C{\mathbb C}
\def\G{\mathbf G}

\def\GL{\mathbf{GL}}
\def \GL2 {{\text{GL}_2}}

\def\Tr{{\rm Tr}}
\def\Gal{{\rm Gal}}
\def\Frob{{\rm Fr}}
\def\F{{\mathbb F}}
  
\def\Z{{\mathbb Z}}

\def\Q{{\mathbb Q}}

\def \R {\mathcal R}
\def\G{\Gamma}

\def\({\left(}
\def\){\right)}
\def \l {\lambda}

\def\CC#1#2{\binom {#1}{#2}}
\def\C{\mathbb{C}}
\def\Z{\mathbb{Z}}
\def\Q{\mathbb{Q}}
\def\F{\mathbb{F}}
\def\Q{\mathbb Q}

\def\({\left(}
\def\){\right)}
\def\G{\Gamma}

\def\CC#1#2{\binom {#1}{#2}}
\def\C{\mathbb{C}}
\def\R{\mathbb{R}}
\def\Z{\mathbb{Z}}
\def\Q{\mathbb{Q}}
\def\F{\mathbb{F}}

\def\H{\mathfrak H}

\def \Frob{\text{Frob}}
\newcommand{\fp}
{\mathbb{F}_p}

\newcommand{\fq}
{\mathbb{F}_q}

\newcommand{\fpc}
{\mathbb{F}_p^{\times}}

\newcommand{\fqhat}
{\widehat{\mathbb{F}_{q}^{\times}}}

\newcommand*\HYPERskip{&}
\catcode`,\active
\newcommand*\pFq{
\begingroup
\catcode`\,\active
\def ,{\HYPERskip}%
\doHyper
}
\catcode`\,12
\def\doHyper#1#2#3#4#5{%
\, _{#1}F_{#2}\left[\begin{matrix}#3 \smallskip \\  #4\end{matrix} \; ; \; #5\right]%
\endgroup
}

\newcommand{\hg}[4]{
{}_{2}F_{1} \left[
\begin{matrix}
#1 & #2 \smallskip \\
   & #3 \\
\end{matrix}
; #4
\right]
}

\newcommand{\hgq}[4]{
\,_{2}F_{1} \left(
\begin{matrix}
#1 & #2 \\
   & #3 \\
\end{matrix}
\, ; #4
\right)_q
}

\newcommand{\hgp}[4]{
_{2}F_{1} \left(
\begin{matrix}
#1 & #2 \\
   & #3 \\
\end{matrix}
\, ; #4
\right)_p
}

\begin{document}

\title{Generalized Legendre curves and Quaternionic Multiplication}

\author{Alyson Deines, Jenny G. Fuselier, Ling Long, Holly Swisher, Fang-Ting Tu}

\address{Center for Communications Research, San Diego, CA 92121, USA}
\email{aly.deines@gmail.com }
\address{High Point University, High Point, NC 27268, USA}
\email{jfuselie@highpoint.edu}
\address{Louisiana State University, Baton Rouge, LA 70803, USA}
\email{llong@math.lsu.edu}
\address{Oregon State University, Corvallis, OR 97331, USA}
\email{swisherh@math.oregonstate.edu}
\address{National Center for Theoretical Sciences, Hsinchu, Taiwan 300, R.O.C.}
\email{ft12@math.cts.nthu.edu.tw}

\begin{abstract}
{This paper is devoted to abelian varieties arising from generalized Legendre curves. In particular, we consider their corresponding Galois representations, periods, and endomorphism algebras. For certain one parameter families of 2-dimensional abelian varieties of this kind, we determine when the endomorphism algebra of each fiber defined over the algebraic closure of $\Q$ contains a quaternion algebra.}
\end{abstract}



\dedicatory{Dedicated to Professor Wen-Ching Winnie Li, an incredible woman, an amazing leader, and who has inspired so many of us.}

\maketitle



\section{Introduction}

Algebraically, for integers $2\leq e_1, e_2, e_3 \leq \infty$, the \emph{triangle
group} $(e_1,e_2,e_3)$ is defined by the presentation
\[
\langle x,y \mid x^{e_1}=y^{e_2}=(xy)^{e_3}=id\rangle.
\]
A triangle group is called \emph{arithmetic} if it has a unique embedding  to $SL_2(\mathbb
R)$ with image either commensurable with $PSL_2(\Z)$, which is isomorphic to $(2,3,\infty)$, or derived from an order of {certain} indefinite quaternion algebra over a totally real field, see \cite[Section 3, Defn. 1]{Takeuchi-triangle}.  Arithmetic triangle groups have been classified  by Takeuchi \cite{Takeuchi-triangle}, \cite{Takeuchi}. Given an arithmetic triangle group
$\G$, its action on the upper half plane, $\mathfrak H$, via linear transformations yields a quotient space.  This space is a modular curve when at least one of $e_i$ is $\infty$; otherwise, it is a Shimura curve.  While modular curves parametrize certain isomorphism classes of elliptic curves, Shimura curves parametrize isomorphism classes of certain 2-dimensional abelian varieties with quaternionic multiplication \cite{Shimura-pav}, \cite{Shimura}. {For this reason, our main result focuses on $2$-dimensional abelian varieties.}  Recently, Yang has computed explicitly automorphic forms on Shimura curves in terms of hypergeometric series \cite{Yang}. Based on that, Yang and the fifth author obtained explicit algebraic transformations of hypergeometric functions \cite{TY}.

Each arithmetic triangle group can be realized as a monodromy group of some {ordinary differential equation (ODE)} satisfied by an integral of the form
\begin{equation}\label{eq:ellpint}
 \int_0^1 \frac{dx}{\sqrt[N]{x^i(1-x)^j(1-\l x)^k}},
\end{equation}
with $N,i,j,k\in \Z$.  Wolfart realized these integrals as periods of the generalized Legendre curves
\[
C_\l ^{[N;i,j,k]}:y^N=x^i(1-x)^j(1-\l x)^k,
\]
where $\l$ is a constant and $N,i,j,k$ are suitable natural numbers \cite{Wolfart}. In this paper, we assume that $1\le i,j,k<N$.  We additionally assume $N\nmid i+j+k$, so that the generic abelian subvarieties that we extract from the smooth models of $C_\l^{[N;i,j,k]}$ do not have complex multiplication (CM). {For a $g$-dimensional abelian variety $A$ over a totally real field $K$, this means the endomorphism of $A$ is a degree $2g$ CM field over $K$.} The CM cases have been previously considered in papers like \cite{Shiga-Tsutsui-Wolfart} and will not be our focus here.

Let $\G$ be an arithmetic triangle group with each  $e_i$ finite.  Then $\G$ corresponds to an explicit quaternion algebra $H_\G$ over a totally real field $K_\G$. Takeuchi gives descriptions of $K_\G$ and $H_\G$ in terms of $e_1,e_2,e_3$ in \cite[Proposition 2]{Takeuchi}. For example, $H_{(3,6,6)}=\left ( \frac{-3,6}{\Q} \right ),$ which is isomorphic to $\left ( \frac{-3,2}{\Q} \right )$. {For any field $K$ whose characteristic is not $2$, $a,b\in K$, the notation $\left (\frac{a,b}{K} \right )$ stands for a quaternion $K$-algebra, which is a 4-dimensional $K$-algebra generated by $1,i,j,ij$ where, $i^2=a,j^2=b, ij=-ji$.}

Assume  {$\gcd(i,j,k)$ is coprime to $N$} and $\l\in \overline \Q$.  Let $J_\l^{[N;i,j,k]}$ be  the Jacobian variety  constructed from the smooth model of $C_\l^{[N;i,j,k]}$. We will consider the primitive part $J_\l^{new}$ of $J_\l^{[N;i,j,k]}$. It is a $\varphi(N)$ dimensional abelian variety, where $\varphi$ is the Euler totient function. We will show the following.

\begin{theorem}\label{thm:1}
Let $N=3,4,6$ and other notation and assumptions as above, in particular, $N\nmid i+j+k$.   Then for each $\l \in \overline{\Q}$, the endomorphism algebra of $J_\l^{new}$  contains a quaternion algebra over $\Q$  if and only if
\[
B \left (\frac{N-i}N,\frac{N-j}N \right )\Big / B \left (\frac kN, \frac{2N-i-j-k}N \right )\in \overline \Q,
\] where $B(a,b)=\int_0^1 x^{a-1}(1-x)^{b-1}dx=\frac{\G(a)\G(b)}{\G(a+b)}$ is the Beta function, and $\G(\cdot)$ is the Gamma function.
\end{theorem}
{The condition $N=3,4,6$ in the above theorem is to guarantee that the abelian varieties obtained are 2-dimensional.

To prove Theorem \ref{thm:1},} we use results of Wolfart \cite{Wolfart} and Archinard \cite{Archinard}, \cite{Archinard-exceptional} to compute the periods of $J_\l^{new}$ and then use a result of  W\"ustholz  \cite{Wu}  on periods and decompositions of abelian varieties. When $N=4,6$, we prove that the Beta quotient is necessarily algebraic by a different method. It involves computing the Galois representations attached to $J_\l^{new}$ via Greene's Gaussian hypergeometric functions \cite{Greene} and a result of Yamamoto \cite{Yamamoto} on Gauss sums.

{Recently Petkova and Shiga \cite{PS}} give a family of 2-dimensional abelian varieties $A'(\l)$ over the Shimura curve for the arithmetic group $(3,6,6)$ such that for each $\l\in \overline \Q$, the endomorphism algebra $End_0(A'(\l))=End(A'(\l))\otimes _{\Z}\Q$ contains  $\left (  \frac{-3,2}{\Q} \right )$. We give a different construction of a one parameter family of abelian varieties with the same property using  the primitive part of $J_\l^{[6;4,3,1]}$.

Our methods apply more generally.  We use them to treat a few other cases in later sections.


\section{Preliminaries}
\subsection{Notation}
Let $F$ be a number field, $\overline F$ its algebraic closure, and $G_F:=\Gal(\overline F/F)$ the absolute Galois group of $F$. The Galois representations in this paper are continuous homomorphisms of $G_F$ to finite dimensional linear groups over fields like $\overline \Q_\ell$ which we assume to be algebraically closed. Such a Galois representation is said to be \emph{strongly irreducible} if its restriction to any finite index subgroup of $G_F$ remains irreducible.
We use $\zeta_N$ to denote a primitive $N$th root of unity. Let $\mathbb{F}_{q}$ denote the finite field of size $q$, $\widehat{\mathbb{F}_{q}^{\times}}$ the group of all multiplicative characters $\chi$ on $\fq^\times$, and $\varepsilon$ the trivial character.  By convention, we let $\chi(0)=0$ so that we can view $\chi$ over $\mathbb{F}_q$.

\subsection{${}_2F_1$-hypergeometric functions and triangle groups}  The ${}_2F_1$-hypergeometric function for $|\l|<1$ is given by
\[
 \pFq{2}{1}{a,b}{,c}{\l}: = \sum_{k=0}^\infty \frac{(a)_k (b)_k}{(c)_k}\frac{\l^k}{k!},
\]
where $(\alpha)_k=\alpha(\alpha+1)\cdots  (\alpha+k-1)$ denotes the Pochhammer symbol.  It is a solution of the Hypergeometric Differential Equation (HDE), denoted by $HDE(a,b;c;\l)$
\[
\l(1-\l)F'' + [(a+b+1)\l-c]F' + abF = 0.
\]
This differential equation is a Fuchsian differential equation with 3 singularities $0,1,\infty$. When $c\notin \Z$, $\l^{1-c}\pFq{2}{1}{1+a-c,1+b-c}{,2-c}{\l}$ is another independent solution of $HDE(a,b;c;\l)$.
{For further details of hypergeometric functions see \cite{AAR}, \cite{Bailey}, \cite{Slater}.}

In this paper we assume $a,b,c\in \Q$.   One can construct a representation of the fundamental group $\pi_1(\C P^1\setminus \{0,1,\infty\}) \rightarrow GL_2(\mathbb C)$ from $HDE(a,b;c;\l)$ which is determined up to conjugation, see \cite{Yoshida}. The image is called the \emph{monodromy group} of  $HDE(a,b;c;\l)$, and is a triangle group.
The following theorem of Schwarz (see \cite{Yoshida}) gives an explicit correspondence between a hypergeometric function and a Schwarz triangle  $\Delta(p,q,r)$ with  $p,q,r \in \Q$. Each $\Delta(p,q,r)$ is the symmetry group of a tiling of either the Euclidean plane, the unit sphere, or the hyperbolic plane (depending on whether $p+q+r$ is equal to, greater than, or less than $1$, respectively) by triangles with angles $p\pi $, $q\pi$, and $r\pi $. We would like to emphasize that ``$\Delta(p,q,r)$" is different from the triangle group $(e_1,e_2,e_3)$ as $e_1,e_2,e_3$ are the denominators of $p,q,r$ when they are written in {lowest terms} as rational numbers.

\begin{theorem}(Schwarz)\label{schwarz}
Let $f,g$ be two independent solutions to $HDE(a,b;c;\l)$ at a point $z\in\H$, and let $p= |1-c|$, $q=|c-a-b|$, and $r=|a-b|$.  Then the Schwarz map $D=f/g$ gives a bijection from $\H\cup \R$ onto a curvilinear triangle with vertices $D(0),D(1), D(\infty)$, and corresponding angles $p\pi,q\pi, r \pi$.
\end{theorem}

\begin{example}
  When $a=\frac 16,b=\frac 13, c=\frac 56$, $p = |1-c|=\frac16, q= |c-a-b|=\frac 13$, $r = |a-b|=\frac 16$. The corresponding triangle group is $(6,3,6)\simeq(3,6,6)$.

\end{example}

\subsection{Generalized Legendre Curves}\label{periods} When $c>b>0$, a formula of Euler (see \cite{AAR}) says
\begin{equation}\label{eq:1}
  \mathcal{P}(\l):=\int_0^1 x^{b-1}(1-x)^{c-b-1}(1-\l x)^{-a}dx= \pFq{2}{1}{a,b}{,c}{\l} \cdot B(b,c-b).
\end{equation}
The restriction  $c>b>0$ is to ensure convergence and can be dropped if we take the Pochhammer contour $\gamma_{01}$ as the integration path \cite{Klein,Yoshida-love}.
If $b$, $c-b \not\in \Z$, which we assume below, then the integral above satisfies
$$
\mathcal{P}(\l) = \, \frac1{\(1-e^{2\pi ib}\)\(1-e^{2\pi i(c-b)}\)}\int_{\gamma_{01}} x^{b-1}(1-x)^{c-b-1}(1-\l x)^{-a}dx,
$$
and is a \emph{period} of an algebraic curve, depending on $a,b,c$. As $(1-e^{2\pi ib})(1-e^{2\pi i(c-b)}) \in \overline \Q$, we will use \eqref{eq:1} to compute the periods below as our criterion lies in the algebraicity of the ratio of two certain periods.
Following Wolfart \cite{Wolfart}, for $\l \neq 0,1$, {the curve  containing this period} can be given by the smooth model of
 $$
    C_\lambda^{[N;i,j,k]}:  y^N=x^i(1-x)^j(1-\lambda x)^k,
 $$
 where $N$ is the least common denominator of $a,b$, and $c$, and the integers $i,j,k$ are defined by $i=N\cdot (1-b)$, $j=N\cdot (1+b-c)$, and $k=N\cdot a$.  Consequently, $\gcd(i,j,k)$ is coprime to $N$.  We call such a curve $C_\lambda^{[N;i,j,k]}$ a {\it generalized Legendre curve}.  It is an $N$-fold cover of $\mathbb C P^1$ which ramifies at  $0,1,\infty$, and $\frac1{\l}$. When $N\mid i+j+k$,  one can rewrite $C_\lambda^{[N;i,j,k]}$ as $y^N=x^i(1-x)^j$ by moving $\frac 1 \l$  to infinity.  In this paper we assume $1\leq i,j,k<N$,  $N\nmid i+j+k$, {and $\gcd(i,j,k)$ is coprime to $N$.}

Using the conversion between $(a,b,c)$ and $(i,j,k,N)$, we can write
\[
\mathcal{P}(\l )=B \left (\frac{N-i}N,\frac{N-j}N\right )\cdot   \pFq{2}{1}{\frac kN, \frac {N-i}N}{,\frac{2N-i-j}N}{\lambda}.
\]
By Theorem \ref{schwarz}, this corresponds to a Schwarz triangle with angles $|\frac{N-i-j}{N}| \pi$, $|\frac{N-j-k}N|\pi$, and $|\frac{N-i-k}N|\pi$.  In this paper, we will exclude the possibility that that the sum of these three angles is $0$ or $\pi$. \footnote{If the sum is $0$, then $i=j=k=\frac{N}2$ which corresponds to the triangle group $(\infty,\infty,\infty)$ with cusps. If the sum is $\pi$, getting rid of the absolute values, we get that the sum is either $|\frac{N-2i}N|\pi$,  $|\frac{N-2j}N|\pi$,  $|\frac{N-2k}N|\pi$, or  $|\frac{3N-2i-2j-2k}N|\pi$. As $1\le i,j,k<N$, the first three possibilities won't occur.  The remaining possibility is that $|\frac{3N-2i-2j-2k}N|\pi=\pi$, which implies that $i+j+k=N$ or $i+j+k=2N$. In either case, $\infty$ is not a singularity for the generalized Legendre curve $C_{\l}^{[N;i,j,k]}$, as we discussed earlier.}

For a fixed curve $C_\lambda^{[N;i,j,k]}$, we let $X(\l):=X_\lambda^{[N;i,j,k]}$ be the desingularization of $C_{\lambda}^{[N;i,j,k]}$. The genus of $X(\l)$ is given by (see \cite{Archinard-exceptional})
\begin{equation}\label{eq:genus}
g(X(\l))=  1+N-\frac{\gcd(N,i+j+k)+\gcd(N,i)+\gcd(N,j)+\gcd(N,k)}2.
\end{equation}
Let $\zeta\in\mu_N$, the multiplicative group of $N$th roots of unity.  For $X(\l)$, the automorphism induced by $A_\zeta:\,(x,y)\mapsto (x,\zeta^{-1} y)$ on $C_\lambda^{[N;i,j,k]}$ plays a central role; it induces a representation of $\Z/N\Z$ on the vector space $H^0(X(\lambda), \Omega^1)$ of the holomorphic differential $1$-forms on $X(\l)$.

\subsection{Holomorphic differential 1-forms on $X_\lambda^{[N;i,j,k]}$} \label{ss:holodiff}
Most of the discussion here can be found in \cite{Archinard-exceptional,Wolfart}. A basis of $H^0(X(\lambda), \Omega^1)$ is given by the regular pull-backs of differentials on $C_\l^{[N;i,j,k]}$ of the  form
\begin{equation}
 \omega= \frac{x^{b_0}(1-x)^{b_1}(1-\lambda x)^{b_2}dx}{y^n}, \quad 0\leq n\leq N-1, \, b_i\in\Z,
\end{equation}
satisfying the following conditions equivalent to the pullback of $\omega$ being regular at $0,1,\frac 1\l, \infty$ respectively.
\begin{align*}
  b_0 \geq \frac{ni+\gcd(N,i)}N-1,\;\;
  b_1 \geq \frac{nj+\gcd(N,j)}N-1,\;\;
  b_2 \geq \frac{nk+\gcd(N,k)}N-1,
  \end{align*}
\begin{equation}\label{eq:6}
   b_0+b_1+b_2 \leq \frac{n(i+j+k)-{\gcd(N,i+j+k)}}{N}-1.
\end{equation}

For each $0\leq n<N$, we let $V_n$ denote the isotypical component of $H^0(X(\lambda), \Omega^1)$ associated to the character $\chi_n:\, \zeta_N\mapsto\zeta_N^n$.  Then the space $H^0(X(\lambda), \Omega^1)$ is decomposed into a direct sum $\displaystyle \bigoplus_{n=0}^{N-1}V_n$.
If $\gcd(n,N)=1$, the dimension of $V_n$ is given by
$
  \dim V_n=\left \{\frac{ni}N\right \}+\left \{\frac{nj}N\right \}+\left \{ \frac{nk}N\right \}-\left \{ \frac{n(i+j+k)}N\right \},
$
where $\left \{ x\right \} =x-\lfloor x \rfloor$ denotes the fractional part of $x$, see \cite{Archinard-exceptional}.  Furthermore,
$$\dim V_n +\dim V_{N-n}=2.$$

The elements of $V_n$ with $\gcd(n,N)=1$ are said to be new. The subspace
 $$
  H^0(X(\lambda), \Omega^1)^{\mbox{new}}=\displaystyle \bigoplus_{\gcd(n,N)=1}V_n
 $$ is of dimension $\varphi(N)$,  Euler's totient function of $N$.

\subsection{Differentials of the second kind on $X_\l^{[N;i,j,k]}$}\label{ss:diff2nd} The de Rham space, denoted by $H_{DR}^1(X(\l),\C)$,  for $X_\l^{[N;i,j,k]}$ is the space of closed differential forms of the second kind on  $X_\l^{[N;i,j,k]}$, modulo exact differentials. It also admits the induced action of $\Z/N\Z$ via the maps $A_\zeta$ on $C_\l^{[N;i,j,k]}$. Thus, comparing with \S\ref{ss:holodiff}, $H_{DR}^1(X(\l),\C)$ also decomposes into eigenspaces $L_n$ associated to the character $\chi_n: \zeta_N \mapsto \zeta_N^n$. {In particular, $V_n$ is a subspace of $L_n$.} When $(n,N)=1$, $\dim L_n=2$. (See \cite{Shiga-Tsutsui-Wolfart} for more details.)

\subsection{Abelian varieties}\label{abelian varieties}
 From the Jacobian of  $X_\l^{[N;i,j,k]}$, one can construct a dimension $g(X(\l))$ (see \eqref{eq:genus}) abelian variety  $J_\l:=J_\l^{[N;i,j,k]}$ defined over $\Q(\l)$. For each $n\mid N$, $J_\l^{[n;i,j,k]}$ is a natural quotient of $J_\l^{[N;i,j,k]}$ and thus $J_\l^{[n;i,j,k]}$ is isomorphic to an abelian subvariety of $J_\l$. Let $J_\l^{new}$ be the primitive part of $J_\l$ so that its intersection with any abelian subvariety isomorphic to $J_\l^{[n;i,j,k]}$ for each $n\mid N$ is zero. Archinard \cite{Archinard} showed that the dimension of $J_\l^{new}$ is $\varphi(N)$ when $N\nmid i+j+k$; {it can also be seen from the discussion in \S \ref{ss:holodiff}.}

\subsection{Gauss sums and the Gamma function}
Here we  will recall some results of Yamamoto \cite{Yamamoto}.
For any Dirichlet character $\chi$ modulo $p$, define the Gauss sum $g(\chi)$ over $\F_p$ by
\[
g(\chi)=\sum_{x\in \fp} \chi(x)\zeta_p^x.
\]
Recall that the Jacobi sum of two characters $\chi_1,\chi_2$ is
\begin{equation}\label{eq:J}
  J(\chi_1,\chi_2)=\sum_{x\in \fp}\chi_1(x)\chi_2(1-x).
\end{equation}When $\chi_1\chi_2\neq \varepsilon$, $J(\chi_1,\chi_2)=\frac{g(\chi_1)g(\chi_2)}{g(\chi_1\chi_2)}$.
(Note that these definitions differ by a sign from those in Yamamoto \cite{Yamamoto}.)  Gauss sums (resp. Jacobi sums) are analogues of the Gamma function (resp. Beta function) over finite fields. For instance, one can compare \eqref{eq:J} with the integral formula for the Beta function. They  also satisfy similar functional equations. When $\chi\in \widehat{ \F_p^{\times}}$
\begin{equation}\label{gauss-norm}
 g(\chi)\overline{g(\chi)}=p, \quad \text{if} \quad \chi \neq \varepsilon,
\end{equation}
which is the analogue of the reflection formula  for $\Gamma(\cdot)$
\begin{equation}\label{reflection}
 \G(z)\G(1-z)=\frac{\pi}{\sin (\pi z)}.
\end{equation}
The Hasse-Davenport relation for Gauss sums can be written as
\begin{equation}\label{H-D}
 g(\chi^{\ell a})=(-1)^{\ell}\chi(\ell ^{\ell a-M/2})\chi(2^{M/2})^{1-\ell }g(\chi^{M/2})^{1-\ell }\prod_{j=0}^{\ell -1}g(\chi^{a+(M/\ell )j})
\end{equation}
where { $a\in \Z$, $M$ is {any} even integer that divides $p-1$, $\ell $ divides $M$, and $\chi$ is an order $M$ Dirichlet character modulo $p$.}   We note that this differs by a factor of $(-1)^{\ell }$ from Yamamoto (see (2) in \cite{Yamamoto}) due to our choices of Gauss and Jacobi sums.  This is an analogue of the multiplication formula for $\Gamma(\cdot)$,
\begin{equation}\label{multiplication}
 \G(\ell z)=\ell ^{(\ell z-\frac12)}2^{\frac{(1-\ell )}{2}}\G\left(\frac 12 \right)^{1-\ell } \cdot \prod_{j=0}^{\ell -1}\G\left(z+\frac j \ell \right).
\end{equation}

\noindent In \cite{Yamamoto}, Yamamoto proved a conjecture of Hasse which says the following:
\begin{theorem}[Yamamoto]\label{Yamamoto}
Let $M\ge 4$ be an even integer, $p\equiv 1\pmod M$ be a prime, then \eqref{gauss-norm} and \eqref{H-D} are the only two relations connecting the Gauss sums $g(\chi)$ for $\chi\in \widehat { \F_p^{\times}}$ satisfying $\chi^M=\varepsilon$, when considered as ideals {in the ring of algebraic integers.}
\end{theorem}
In \cite{Yamamoto2}, Yamamoto classified additional relations between Gauss sums due to sign ambiguity, {which do not occur in any of our settings}.

 In comparison, for any integer $i$ coprime to $M$ and any $\chi\in  \widehat { \F_p^{\times}}$ of order $M$, we have the following ``{analogy}''
\[
\begin{array}{cccc}
 \frac{i}M&\Longleftrightarrow  &\chi^i\\
\frac12 &\Longleftrightarrow&  \chi^{M/2}\\
\G(\frac{i}M)&\Longleftrightarrow&g(\chi^i)\\
B(\frac{i}M,\frac{j}M) &\Longleftrightarrow& J(\chi^i,\chi^j).\\
\end{array}
\]

\begin{lemma}\label{prop:gauss-Beta}
 Let $M\ge 4$ be an even integer and $i,j,k\in \Z$ such that $M$ does not divide any of $i,j,k,k-i,k-j$. If for any prime {$p\equiv 1 \pmod M$} and any $\eta\in  \widehat { \F_p^{\times}}$ of order $M$, the quotient $F(\eta):=J(\eta^j,\eta^{k-j})/J(\eta^i, \eta^{k-i})$ is a character, { by which we mean it is a homomorphism from the group of order $M$ characters in $\widehat{\F}_p$ to $\C^{\times}$,}  then $B(\frac{j}M, \frac{k-j}M)/B(\frac{i}M, \frac{k-i}M)$ is an algebraic number.
 \end{lemma}

{ We remark that in Lemma \ref{prop:gauss-Beta}, a particularly convenient choice for $\eta$ is the Teichmuller character raised to the power $-\frac{p-1}M$, i.e. $\eta(x)\equiv x^{-\frac{p-1}M} \pmod{p}$ for all $x\in \F_p$. See \cite{WIN3b} for more details.}

\begin{proof}
By Theorem \ref{Yamamoto}, $J(\eta^j,\eta^{k-j})/J(\eta^i, \eta^{k-i})$ can be reduced to a character using only the relations \eqref{gauss-norm} and \eqref{H-D}. By \eqref{H-D}, when
$\displaystyle
\frac{g(\chi^{\ell a})}{g(\chi^{M/2})^{1-\ell }\prod_{j=0}^{\ell -1}g(\chi^{a+(M/\ell )j})}
$
is a character, then  \eqref{multiplication} says that
$\displaystyle
\frac{\G(\frac{\ell a}M)}{\G(\frac 12)^{1-\ell }\prod_{j=0}^{\ell -1}\G(\frac aM+\frac j \ell )}
$
is algebraic.  Thus if \eqref{gauss-norm} is needed to reduce $J(\eta^j,\eta^{k-j})/J(\eta^i, \eta^{k-i})$ to a character which is of norm one, then it will be used equally many times on the denominator and numerator. Thus, one will use the reflection formula \eqref{reflection} equally many times on the denominator and numerator of $B(\frac{j}M, \frac{k-j}M)/B(\frac{i}M, \frac{k-i}M)$, which yields that the Beta ratio is algebraic.
\end{proof}

\begin{example}\label{ex:[10,1,6,5]}
 Let $p\equiv 1 \pmod{10}$ be prime and $\eta \in \widehat{\F_p^\times}$  of order $10$. Then
$$
   J(\eta,\eta^6)/J(\eta^2,\eta^5)=\eta(-1)J(\eta,\eta^5)/J(\eta^2,\eta^4)=\eta^8(2).
$$
This can be deduced from  \cite[Chapter 3]{Berndt-Evans-Williams} or from \eqref{gauss-norm} and \eqref{H-D}.
In comparison
$$
  B\(\frac 1{10},\frac 6{10}\) \big /B\(\frac 2{10},\frac 5{10}\)=2^{\frac45}.
$$
\end{example}


\section{Counting points on generalized Legendre curves} \label{counting}

\subsection{Preliminaries}

\noindent Let $p$ be an odd prime and $q=p^s.$  We recall the $_2F_1$ hypergeometric function over $\fq$:

\begin{defn}[\cite{Greene}, Def. 3.5]\label{hg}
For  $A$, $B$, and $C$ in $\widehat{\mathbb{F}_{q}^{\times}}$ and $\lambda\in\fq$, define
$$\hgq{A}{B}{C}{\lambda}:=\varepsilon(\lambda) \frac{BC(-1)}{q} \sum_{x\in\fq} B(x)\overline{B}C(1-x)\overline{A}(1-\lambda x).$$
\end{defn}

\begin{theorem}[\cite{Greene}, Thm. 3.6]\label{other 2F1}
{Assume notation as above. Then,}
$$\hgq{A}{B}{C}{\lambda}
= \frac{q}{q-1} \sum_{\chi\in\widehat{\mathbb{F}_{q}^{\times}}} \binom{A\chi}{\chi} \binom{B\chi}{C\chi} \chi(\lambda),$$ where $\displaystyle \binom{A}{B}:=\frac{B(-1)}{q}J(A,\overline{B})=\frac{B(-1)}{q} \sum_{x\in \fq} A(x)\overline{B}(1-x).$
\end{theorem}

\begin{lemma}\label{2f1_jacobi1}  If $A, B, A\overline{C}, B\overline{C}\neq \varepsilon$
 $$
   J(A,\overline C)\cdot \hgq ABC\lambda=J(B,\overline C)\cdot \hgq BAC\lambda.
$$
\end{lemma}
\begin{proof}
 By Theorem \ref{other 2F1},
$$
   \hgq{A}{B}{C}{\lambda}=\frac{C(-1)}{q(q-1)} \sum_{\chi\in \widehat{\F_q^\times}}J(A\chi,\overline\chi)J(B\chi,\overline {C\chi})\chi(\lambda).
$$
Then the Lemma follows from the relation between Jacobi and Gauss sums, in particular
$$
  J(A,\overline C)g(B)g(A\overline C)=J(B,\overline C)g(A)g(B\overline C).
$$
\end{proof}
\noindent As an immediate consequence when $C=\varepsilon$, we have the following corollary.

\begin{corollary}\label{cor:8}
For any $A,B\in \widehat{\F_q^\times}$,
$$
   \hgq AB\varepsilon\lambda=\, \hgq {B}{A}{\varepsilon}\lambda.
$$
\end{corollary}

We are now ready to prove one of the key results that we will use later.
\begin{prop}\label{prop:9}
If $A, B, C \in \widehat{\F_q^\times}$, $A,B\neq \varepsilon$, $A,B\neq C$, and $\l\in \fq\setminus \{ 0,1\}$, then
\begin{equation*}
\hgq ABC\lambda=
AB(-1)\overline C(-\lambda)C\overline{AB}(1-\lambda)\frac{J(B,\overline B C)}{J(A,\overline A C)}\hgq {\overline A}{\overline B}{\overline C}\lambda.
\end{equation*}
\end{prop}

\begin{proof}
The key ingredients of this proof  are the transformation formulas
Theorem 4.4 parts (i) and  (iv) in Greene's paper \cite{Greene} and the identity in Lemma \ref{2f1_jacobi1}.  Greene's formulas are stated as
\begin{align*}
   \hgq ABC\lambda&= A(-1)\hgq AB{AB\overline C}{1-\lambda}\\
            &\qquad +A(-1)\CC{B}{\overline AC}\delta(1-\l)-\CC BC\delta(\l)\\
   \hgq ABC\lambda&= C(-1)C\overline{AB}(1-\l)\hgq {C\overline A}{C\overline B}{C}{\lambda}\\
    &\qquad +A(-1)\CC{B}{\overline AC}\delta(1-\l),
\end{align*}
where
$$
  \delta(x)=
  \begin{cases}
    1,\, & \mbox{ if }  x=0,\\
    0,\, & \mbox{ if }  x\neq0.
  \end{cases}
$$

For $\lambda\neq0,1$, from the transformation formulas, we have
  \begin{align*}
     \hgq ABC\lambda&= A(-1)\hgq AB{AB\overline C}{1-\lambda}\\
             &=B\overline{C}(-1)\overline C(\lambda)\hgq {B\overline C}{A\overline C}{AB\overline C}{1-\lambda}\\
             &=\overline C(\lambda)\hgq {B\overline C}{A\overline C}{\overline C}{ \l }\\
             &= \overline C(-\lambda)\overline{AB}C(1-\lambda)\hgq{\overline B}{\overline A}{\overline C}{\lambda}.
  \end{align*}
By Lemma \ref{2f1_jacobi1} and the fact $J(A,\overline C)J(\overline A,C)=J(B,\overline C)J(\overline B,C)=q$, one has
$$
  J(A,\overline C)\hgq ABC\lambda=\overline C(-\lambda)C\overline{AB}(1-\lambda)J(B,\overline C)\hgq {\overline A}{\overline B}{\overline C}\lambda.
$$
Our claim follows from the fact that
$$
  J(\chi,\psi\overline \chi)=\chi(-1)J(\chi,\overline \psi), \mbox{ if } \psi, \chi\neq \varepsilon.
$$
\end{proof}

\subsection{Desingularization for $C_{\lambda}^{[N;i,j,k]}$}
Recall that,  $C_{\lambda}^{[N;i,j,k]}$ has $4$ singularities: $0, 1,\infty,\frac1{\l}$. Upon choosing a suitable uniformizer $s$ at each singularity, the curve is of the form $t^N=s^nf(s)$ with $f(s)\in \Q(\l)[s], f(0)\neq 0$ where $t$ is a rational function of $s$ and $y$. For instance {near zero, letting $s=x$ and $t=y$,} it is $t^N=s^if(s)$, with $f(s)=(1-s)^j(1-\l s)^k$. Similarly, near infinity, we let $s=\frac 1x$ and  $t=\frac{y}{s^e}$ where $e=\left [ \frac{i+j+k}N \right ]+1$ to obtain $t^N=s^{n}(s-1)^j(s-\l)^k$ with $n=N+N \left [ \frac{i+j+k}N \right ]-i-j-k$. (We omit the similar discussion for the other two singularities $1,\frac 1 \l$.)

Now we will resolve the singularity of $t^N=s^nf(s)$ at $s=0$. By the division algorithm, there are two integers $u,v$ such that $Nu-nv=\gcd(N,n)$. Letting $t=w^u$ and $s=w^v$, $t^N=s^nf(s)$ can be rewritten as $w^{\gcd(N,n)}=f(w^v)$. Over a finite field $\fq$ and when $s=0$, i.e. $w=0$, the point $(s,y)=(0,0)$ is resolved to $\displaystyle \#\{w\in \fq \mid w^{\gcd(N,n)}=f(0)\}$ points.  We let $n_0,n_1,n_\infty,n_{\frac1\l}$ denote this number at the corresponding singularities.  For instance, $\displaystyle  n_0=\#\{w\in \fq :  w^{\gcd(N,i)}=1\}.$

\subsection{Counting points on $X_{\lambda}^{[N;i,j,k]}$ over finite fields}

There are multiple ways in which one can use Gaussian hypergeometric functions to count points on varieties, and these functions are also related to coefficients of various modular forms including Siegel modular forms \cite{Ahlgren-Ono, Barman, FOP, Fuselier, Koike, Lennon-count, MP, Ono, Ve11}. For our purposes, we use a technique similar to one shown in \cite{Ve11}.  First, recall the following well-known and useful lemma:

\begin{lemma}[\cite{IR90}, Prop. 8.1.5]\label{count_lemma}
  If $a\in\fq^{\times}$ and $n\mid(q-1),$ then $$\#\{ x\in\fq : x^n=a \} =\sum_{\chi^n=\varepsilon}\chi(a),$$ where the sum runs over all characters $\chi\in\fq ^\times$ whose order divides $n$. \end{lemma}

Also, note that
\begin{equation}\label{eq:n0}
n_0-1=\sum_{m=1}^{\gcd(N,i)-1}\chi^m(1),
\end{equation}
where $\chi$ is any order $\gcd(N,i)$ character of $\fq ^\times$. {There are similar formulas for $n_1-1,n_{\frac1\l}-1,n_\infty-1$.}

\begin{theorem}\label{main count}
Let $p>3$ be prime and $q=p^s\equiv 1\pmod{N}$, and let $i,j,k$ be natural numbers with $1\leq i,j,k <N$. Further, let $\xi\in\fqhat$ be a character of order $N$. Then for $\lambda\in \fq\setminus \{0,1\} ,$

\begin{multline*}
   \# X_{\lambda}^{[N;i,j,k]}  (\fq)=1+q+q \sum_{m=1}^{N-1} \xi^{mj}(-1) \hgq{\xi^{-km}}{\xi^{im}}{\xi^{m(i+j)}}{\lambda}\\\hspace*{-1.5in}+n_0+n_1+n_{\frac1\l}+n_\infty-4,
 \end{multline*}
where $n_0,n_1,n_{\frac1\l},n_\infty$ are as before.
\end{theorem}

\begin{proof}

At infinity, there are $n_\infty$ points which correspond to $\frac 1x=0$.  There are an additional $n_0+n_1+n_{\frac1\l}-3$ points from other singularities. Thus,
\begin{align*}
 \#X_{\lambda}&^{[N;i,j,k]} (\fq)  \\
 =&\sum_{x\in\fq} \#\{y\in\fq : y^N=x^i(1-x)^j(1-\lambda x)^k\}+n_\infty+n_0+n_1+n_{\frac1\l}-3\\
=& \sum_{{x\in\fq}} \#\{y\in\fq^{\times} : y^N=x^i(1-x)^j(1-\lambda x)^k\} \\
& + \# \{ x\in\fq : x^i(1-x)^j(1-\lambda x)^k=0\}+n_\infty+n_0+n_1+n_{\frac1\l}-3.\\
=& \sum_{x\in\fq} \#\{y\in\fq^{\times} : y^N=x^i(1-x)^j(1-\lambda x)^k\}  +n_\infty+n_0+n_1+n_{\frac1\l}.
\end{align*}

Now we apply Lemma \ref{count_lemma} to the first sum, recalling that $\xi$ is of order $N$.
\begin{multline*}
 \#X_{\lambda}^{[N;i,j,k]}  (\fq) = \sum_{x\in\fq}\sum_{m=0}^{N-1} \xi^m(x^i(1-x)^j(1-\lambda x)^k)  +n_\infty+n_0+n_1+n_{\frac1\l}\\
 =  \sum_{x\in\fq} \varepsilon(x^i(1-x)^j(1-\lambda x)^k) + \sum_{x\in\fq}\sum_{m=1}^{N-1} \xi^m(x^i(1-x)^j(1-\lambda x)^k) \\
+n_\infty+n_0+n_1+n_{\frac1\l}\\
 =  q-3 + \sum_{x\in\fq}\sum_{m=1}^{N-1} \xi^m(x^i(1-x)^j(1-\lambda x)^k) +n_\infty+n_0+n_1+n_{\frac1\l}\\
=1+q + \sum_{x\in\fq}\sum_{m=1}^{N-1} \xi^m(x^i(1-x)^j(1-\lambda x)^k) +n_\infty+n_0+n_1+n_{\frac1\l}-4.
\end{multline*}

Meanwhile, using Definition \ref{hg}, we have

\begin{eqnarray*}
&&q \cdot \, \hgq{\xi^{-km}}{\xi^{im}}{\xi^{m(i+j)}}{\lambda} \\& =& \varepsilon(\lambda) \xi^{m(2i+j)}(-1)\sum_{x\in\fq} \xi^{im} (x)\overline{\xi^{im}} \xi^{m(i+j)} (1-x) \overline{\xi^{-km}} (1-\lambda x) \\
& = &\varepsilon(\lambda) \xi^{mj}(-1)\sum_{x\in\fq} \xi^{m} (x^i(1-x)^j (1-\lambda x)^k).
\end{eqnarray*}
\end{proof}

\subsection{Galois representations}\label{ss:Galois 3.5}  For simplicity, we assume $\l\in \Q$ here. One can construct a compatible family of degree-$2g$ representations
\begin{equation}\label{representation}
\rho: G_{\Q}:=Gal(\overline \Q/\Q)\rightarrow GL_{2g} (\overline \Q_\ell)
\end{equation}
via the Tate module of the Jacobian $J_\l^{[N;i,j,k]}$ of $X_\lambda^{[N;i,j,k]}$.  We now let $\ell$ be a fixed prime.  Recall that $g$ is the genus of $X_\lambda^{[N;i,j,k]}$ as in \eqref{eq:genus}.  We have the lemma below which follows from the relation between the Galois representations and local L-function{s}  for algebraic curves.  Let $p$ be a prime unramified for $\rho$, $q=p^s$ be a power of $p$.  For the conjugacy class of geometric Frobenius $\Frob_q$ in $G_\Q$,
$$-\Tr\rho(\Frob_q)=\sum_{m=1}^{N-1}\sum_{x\in\fq} \xi^m(x^i(1-x)^j(1-\lambda x)^k) +n_\infty+n_0+n_1+n_{\frac1\l}-4.$$ From previous discussions,  $n_0-1=\sum_{m=1}^{\gcd(N,i)-1}\left (\xi^{N/\gcd(N,i)}\right )^m(1)$. We can write $n_\infty-1, n_1-1,n_{\frac1\l}-1$ similarly. Thus we can write $$-\Tr\rho(\Frob_q)=\sum_{m=1}^{N-1}\sum_{x\in\fq}\, ^* \xi^m(x^i(1-x)^j(1-\lambda x)^k),$$ where $*$ means we {disperse the formulas for $n_0-1, n_\infty-1, n_1-1,n_{\frac1\l}-1$ like \eqref{eq:n0}} to summands  corresponding to the characters involved. We  assume that  $N\nmid i+j+k$.
\begin{lemma}\label{lem:10}
 When $(m,N)=1$,
\begin{multline}
  \sum_{x\in\fq}\, ^* \xi^m(x^i(1-x)^j(1-\lambda x)^k)=\sum_{x\in\fq} \xi^m(x^i(1-x)^j(1-\lambda x)^k)\\=q\varepsilon(\l)\xi^{mj}(-1)\cdot \,\hgq{\xi^{-km}}{\xi^{im}}{\xi^{m(i+j)}}{\lambda}.
\end{multline}
\end{lemma}

Each map $A_\zeta:(x,y)\mapsto (x,\zeta^{-1} y)$ on $C_\l^{[N;i,j,k]}$ induces an automorphism on $X_\l^{[N;i,j,k]}$, $J_\l^{[N;i,j,k]}$ as well as the representation space $W$ of $\rho$, which will be denoted by $A_\zeta^*$ for simplicity. Thus $W$ also decomposes into isotypic spaces $W_m$ of $A_\zeta^*$ for each character $\chi: \zeta_N\mapsto \zeta_N^m$. As $A_\zeta^*$ is defined over $\Q(\zeta_N)$,  the restriction of $\rho$ to $H=Gal(\overline \Q/\Q(\zeta_N))$  decomposes into a direct sum $\displaystyle \oplus_{m=1}^{N-1} \sigma_m$ where $W_m$ is the representation space of  $\sigma_m$ respectively.  For any proper divisor $n$ of $N$, the corresponding Jacobian $J_\l^{[n;i,j,k]}$ is a quotient of $J_\l^{[N;i,j,k]}$. Thus, the Galois representation for $J_\l^{[n;i,j,k]}$ is a subrepresentation of that for $J_\l^{[N;i,j,k]}$. There is a degree $2\varphi(N)$ subrepresentation $\rho^{new}$ of $\rho$ that corresponds to $J_\l^{new}$. Assuming  $End_0(J_\l^{new})$ is not a CM field, we know $\rho^{new}=\text{Ind}_{G_{\Q(\zeta_N)}}^{G_\Q} \sigma_1$ is an induction with $\sigma_1$ being absolutely irreducible for almost all primes $\ell$. It is enough to fix one such $\ell$ to proceed.

{Below we will separate the sub-representations $\sigma_m$ by twisting the corresponding generalized Legendre curves.}
\begin{prop}\label{prop:14}
Given the notation above, suppose $\l\in \Q$ where $End_0(J_\l^{new})$ is not a CM field and is of degree $2\varphi(N)$, $p$ is unramified for $\rho$ such that $\l\not\equiv 0, 1 \pmod p$, and $p$ splits in $\Q(\zeta_N)$. Then, when $(m,N)=1$, the values
$$\Tr {\sigma_m(\Frob_q)} \quad \text{and}  \quad
-\sum_{x\in\fq}\, ^* \xi^m(x^i(1-x)^j(1-\lambda x)^k)$$ agree up to different embeddings of $\Q(\zeta_N)$ in $\C$.
\end{prop}

\begin{proof}
Assume $q=p^s$ is fixed. Let $c\in \Q$  such that $x^N=c$ is irreducible over $\F_q$. Let $F=\Q(\zeta_N)(\sqrt[N]{c})$ be a degree-$N$  extension of $\Q(\zeta_N)$.

Now consider the family of twisted curves $$X_{m}(\l): \quad y^N=c^m\cdot x^i(1-x)^j(1-\l x)^k,$$ and denote $X_\l^{[N;i,j,k]}$ by  $X_0(\l)$. Over $F$, $X_{m-1}(\l)$ is isomorphic to $X_{m}(\l)$  by the map $T:(x,y)\mapsto (x,\frac{ y}{\sqrt[N]{c}})$. Let $\rho_{m}$ denote the $\ell$-adic representation corresponding to $X_{m}(\l)$ over $\overline{\Q}_\ell$. It has a degree-2$\varphi(N)$ primitive part $\rho_{m}^{new}$, which is also induced from a degree-2 representation $\sigma_{1,m}$ of $G_{\Q(\zeta_N)}$. We can assume $\sigma_{1,m}$ is strongly irreducible. As $\sigma_{1,m}$ and $\sigma_{1}$ are both strongly irreducible and are isomorphic when they are restricted to $G_F$, they differ by a character $\chi_m$ of $G_{\Q(\zeta_N)}$ with kernel $G_F$ (see \cite{Clifford}, Theorem 5). By the choice of $c$, $\chi_m(\Frob_q)$ is a root of unity denoted by $\mu_{N,m}$; it is primitive when $\gcd(m,N)=1$. The map $T$ induces an intertwining operator between $\rho_{m-1}^{new}$ and $\rho_{m}^{new}$. Thus for any $v\in W_1$ the representation space of $\sigma_1$,  $T(\sigma_1(\Frob_q) \cdot v)= \mu_{N,1}\sigma_1(\Frob_q)\cdot (T v)$ and  $T^2(\sigma_1(\Frob_q)\cdot v)= T( \mu_{N,1}\sigma_1(\Frob_q)\cdot Tv))= \mu_{N,1}T(\sigma_1(\Frob_q)\cdot Tv))=\mu_{N,1}^2\sigma_1(\Frob_q)\cdot (T^2 v)$, so $\sigma_{1,2}(\Frob_q)=\mu_{N,1}^2\sigma_1(\Frob_q)$. By induction, $\mu_{N,m}=\mu_{N,1}^m$.

From  the counting points on finite field perspective,
\begin{align*}
-\Tr \rho_{m}^{new}(\Frob_q)&=\sum_{(n,N)=1} \sum_{x\in\fp} \xi^n(c^mx^i(1-x)^j(1-\l x)^k)\\&=\sum_{(n,N)=1} (\xi^{n}(c))^m\sum_{x\in\fq} \xi^n(x^i(1-x)^j(1-\lambda x)^k).
\end{align*}
Since $\xi$ has order $N$ and by the choice of $c$, $\xi(c)$ is a primitive $N$th root of unity. Meanwhile, we can also compute the same quantity using  $\rho_{m}^{new}$ being induced from $\sigma_{1,m}$. Combining them, we get a system of $\varphi(N)$ equations when we let $m\in (\Z/N\Z)^{\times}$. Writing the system in the matrix form, we have on one hand a $\varphi(N)\times \varphi(N)$ matrix $(\xi(c)^{mn})_{m,n \in (\Z/N\Z)^\times}$ times a column vector with entries $-\sum_{x\in\fp} \xi^n(x^i(1-x)^j(1-\lambda x)^k)$, with $(n,N)=1$; on the other hand, we have another  $\varphi(N)\times \varphi(N)$ matrix $(\mu_{N,1}^{mn})_{m,n\in (\Z/N\Z)^\times}$ times a column vector with entries $\Tr \sigma_n(\Frob_p)$. Our claims follow from the fact that the matrices are both invertible and are different by a permutation matrix.
\end{proof}

\begin{corollary}
Given assumptions as above,  $$\Tr {\sigma_m(\Frob_q)}\quad \text{and} \quad -\hgq{\xi^{-km}}{\xi^{im}}{\xi^{m(i+j)}}{\lambda}  \cdot \xi^{mj}(-1)q$$ agree up to different embeddings of $\Q(\zeta_N)$ in $\C$.
\end{corollary}

As $N\nmid i+j+k$, for a generic choice of $\l$ in a totally real field both $\sigma_m$ and $\sigma_{N-m}$ are strongly  irreducible.  Below, we summarize a result in \cite{ALLL} on 4-dimensional Galois representations with Quaternionic Multiplication. Previous discussions on this topic include \cite{ALL,HLV}.
\begin{prop}\label{prop:15} Let $\varphi(N)=2$, with previous assumptions and notation.
  If $End_0(J_\l^{new})$ contains a quaternion algebra, then the corresponding representations $\sigma_1$ and $\sigma_{N-1}$ of $G_{\Q( \zeta_N)}$, which are assumed to be absolutely irreducible, differ by a character.
\end{prop}
\begin{proof}
Corresponding to $J_\l^{new}$ is a compatible family of 4-dimensional $\ell$-adic  representations $\pi$ of $G_{\Q}$ such that $\pi|_{G_{\Q(\zeta_N)}}=\sigma_1\oplus \sigma_{N-1}$.  If $End_0(J_\l^{new})$ contains a quaternion algebra, {from the discussion below in section \ref{ss:PS} we can see that} there are two endomorphisms $I,J$ of $J_\l^{new}$ such that $I^2$ and $J^2$ are  scalars and $IJ=-JI$.  Following the discussion in Section 3 of \cite{ALLL}, if $F$ is a Galois extension of $\Q$ over which both associated actions $I$ and $J$ on $W$ are defined, then $\sigma_1|_{G_F}$ and $\sigma_{N-1}|_{G_F}$ are isomorphic. As we assume both $\sigma_n$ and $\sigma_{N-n}$ are strongly irreducible, $\sigma_1$ and $\sigma_{N-1}$ differ by a character of $G_{\Q(\zeta_N)}$, following a result of Clifford \cite[Theorem 5]{Clifford}.
\end{proof}

\begin{prop}\label{prop:17}
  Assume that $N\nmid i+j+k$ and $N=4,6$.  If  for any $\l\in \Q$, $\sigma_1$ and $\sigma_{N-1}$ differ by a character, then
$B\(\frac{-k}N,\frac{(i+j+k)}N\)\big/B\(\frac{i}N,\frac{j}N\)\in \overline \Q$.
\end{prop}
\begin{proof}
Let $p>5$ be a prime that is congruent to 1 modulo $N$. Then by the previous two results and Proposition \ref{prop:9}, for any $\eta\in \widehat {\fp ^\times}$ of order $N$, $J(\eta^{-k},\eta^{(i+j+k)})/J(\eta^{i},\eta^{j})$ has to be a character. Our claim follows from Proposition \ref{prop:gauss-Beta}.
\end{proof}

\section{Proof of the Theorem \ref{thm:1} and remarks}
\subsection{Proof of the Theorem \ref{thm:1}}
As we have seen in \S \ref{abelian varieties}, when $\varphi(N)=2$, $J_\l^{new}$ is a 2-dimensional abelian variety over $\Q(\l)$.  We now consider when $End_0(J_\l^{new})$ contains a quaternion algebra for each $\l\in \overline \Q$.

As $C_\lambda^{[N;i,j,k]}$ is an $N$-fold cover of $\mathbb{C} P^1$ generically with $4$ singularities $0,1, \frac 1 \l, \infty$, one will pick up two independent Pochhammer contours $\gamma_{01}$ and $\gamma_{\frac1{\lambda}\infty}$. The corresponding periods are algebraic multiples of
$\displaystyle
\int_{\gamma_{01}}\omega$ and $\displaystyle  \int_{\gamma_{\frac1{\lambda}\infty}}\omega$, where $\omega$ is a differential of the first or second kind on $X(\l):=X_{\l}^{[N;i,j,k]}$ \cite{Yoshida-love}.

{Recall our discussion of holomorphic differentials on $X(\l)$ in \S \ref{periods}. Let ${x^{b-1}(1-x)^{c-b-1}(1-\lambda x)^{-a}dx}$ be any differential $\omega$ on $X(\l)$. We are going to integrate it along two different cycles, denoted by $\gamma_{01}$ and $\gamma_{\frac 1\l \infty}$, corresponding to the two following integrals respectively}

\begin{align}\label{eq:int1}
  \int_{0}^1\omega=B(b,c-b) \pFq{2}{1}{a & b}{&c}{\l},
\end{align}
\begin{multline}\label{eq:int2}
  \int_{\frac 1\lambda}^\infty\omega=(-1)^{{c-a-b-1}}\lambda^{1-c}\int_0^1x^{a-c}(1-x)^{-a}(1-\lambda x)^{c-b-1}dx\\
             =(-1)^{{c-a-b-1}}\lambda^{1-c}B(1+a-c,1-a)\pFq{2}{1}{1+b-c & 1+a-c}{& 2-c}{\l},
\end{multline}
which are two independent solutions of the $HDE(a,b;c;\l)$ when $c$ is not an integer.

When $N=3,4,6$, $(\Z/N\Z)^\times=\{1, N-1\}$.
In this case, $$\omega_1=\frac{dx}{y}= \frac{dx}{\sqrt[N]{x^i(1-x)^j(1-\l x)^k}} \in L_1$$  $$\omega_{N-1}:=\frac{x^{i-1}(1-x)^{j-1}(1-\l x)^{k-1}dx}{y^{N-1}}=\frac{dx}{\sqrt[N]{x^{N-i}(1-x)^{N-j}(1-\l x)^{N-k}}}\in L_{N-1}$$ are differentials of $J_\l^{new}$ of the second kind.

Thus letting $a=\frac{k}N, b=\frac{N-i}N, c=\frac{2N-i-j}N$ in \eqref{eq:int1} and \eqref{eq:int2} we get
\begin{align*}
 & \tau_1=\int_{0}^1\omega_1=B\(\frac{N-i}N,\frac{N-j}N\)\hg {\frac{k}N}{\frac{N-i}N} {\frac{2N-i-j}N}\lambda,\\
 & \tau_1'=\int_{\frac 1\lambda}^\infty\omega_1=(-1)^{{-\frac{k+j}N}}\lambda^{\frac{i+j-N}N}B\(\frac{i+j+k-N}N,\frac{N-k}{N}\)\hg {\frac{j}N}{\frac{i+j+k-N}N}{\frac{i+j}N}\lambda, \\
& \tau_{N-1}=\int_{0}^1\omega_{N-1}=B\(\frac iN,\frac jN\)\hg {\frac{N-k}N}{\frac iN}{\frac{i+j}N}\lambda,\\
 & \tau_{N-1}'=\int_{\frac 1\lambda}^\infty\omega_{N-1}\\
& \quad \quad =  (-1)^{{\frac{k+j}N}}\lambda^{\frac{N-i-j}N}B\(\frac{2N-i-j-k}N,\frac kN\)\hg {\frac{N-j}N}{\frac{2N-i-j-k}N}{\frac{2N-i-j}N}\lambda.
\end{align*}
By Euler's transformation formula
{
$$
  \pFq{2}{1}{a,b}{,c}{\lambda}=(1-\l)^{c-a-b}\hg{c-a}{c-b}{c}\l
$$
}
(see \cite[Theorem 2.2.5]{AAR}), $\tau_1'$ is related to $\tau_{N-1}$:
$$
  \tau_1' =(-1)^{-\frac{k+j}N}\lambda^{\frac{i+j-N}N}(1-\l)^{\frac{N-j-k}N}B\(\frac{i+j+k-N}N,\frac{N-k}{N}\)\hg {\frac{N-k}N}{\frac iN}{\frac{i+j}N}\lambda.
$$
Similarly $\tau_{N-1}'$ is related to $\tau_1$ as
$$
   \tau_{N-1}'=(-1)^{{\frac{k+j}N}}\lambda^{\frac{N-i-j}N}(1-\l)^{\frac{k+j-N}N}B\(\frac{2N-i-j-k}N,\frac kN\)\hg {\frac{k}N}{\frac{N-i}N}{\frac{2N-i-j}N}\lambda.
$$

\noindent Thus we have
\begin{align*}
 & \tau_{N-1}'/\tau_1=\alpha(\lambda)\frac{\Gamma\(2-\frac{i+j+k}N\)\Gamma\(\frac kN\)}{\Gamma\(1-\frac{i}N\)\Gamma\(1-\frac{j}N\)}, \,
  \tau_1'/\tau_{N-1}=\frac{\Gamma\(\frac{i+j+k}N-1\)\Gamma\(1-\frac kN\)}{\alpha(\lambda)\Gamma\(\frac{i}N\)\Gamma\(\frac{j}N\)},
\end{align*}
where $\alpha(\lambda)=(-1)^{\frac{k+j}N}\lambda^{\frac{N-i-j}N}(1-\lambda)^{\frac{k+j-N}N}$.
Therefore,
\begin{equation}\label{eq:18}
  \gamma= {\frac{\tau_1'\tau_{N-1}'}{\tau_1\tau_{N-1}}}=\frac{{\(\sin \frac{i}N\pi\)\(\sin \frac{j}N\pi\)}}{\(\sin \frac kN\pi\)\(\sin\frac{2N-i-j-k}N\pi\)}\, \in \Q(\zeta_N+\zeta_N^{-1}).
\end{equation}

\begin{theorem}(W\"{u}stholz \cite{Wu})\label{wu}
  Let $A$ be an abelian variety isogenous over $\overline\Q$ to the direct product $A_1^{n_1}\times\cdots \times A_k^{n_k}$ of simple, pairwise non-isogenous abelian varieties $A_\mu$ defined over $\overline\Q$, $\mu=1,\ldots,k$. Let $\Lambda_{\overline\Q}(A)$ denote the space of all periods of differentials, defined over $\overline\Q$, of the first kind and the second on $A$. Then the vector space $\widehat{V}_A$ over $\overline{\Q}$ generated by $1$, $2\pi i$, and $\Lambda_{\overline\Q}(A)$, has dimension
  $$
    \dim_{\overline\Q}\widehat{V}_A=2+4\sum_{\nu=1}^k\frac{\dim A_\nu^2}{\dim_{\Q}(\mbox{End}_0A_\nu)}.
  $$
 \end{theorem}
\noindent This version of W\"{u}stholz's result is given by Cohen in \cite[Appendix]{Shiga-Tsutsui-Wolfart}.

Assume that $\l\in \overline \Q$  and  $J_\l^{new}$  does not admit CM. In this case, Theorem \ref{wu} implies that  for the 2-dimensional abelian variety $J_\l^{new}$,  its periods of the first and second kinds plus 1 and $2 \pi i$ span the vector space of dimension 10 (or 6) over $\overline \Q$.  When $\frac{\Gamma\(2-\frac{i+j+k}N\)\Gamma\(\frac kN\)}{\Gamma\(1-\frac{i}N\)\Gamma\(1-\frac{j}N\)}$ is {an algebraic number}, the dimension is less than 9 as $\tau_1$ and $\tau_{N_1}'$ (resp. $\tau_1'$ and $\tau_{N-1}$) are algebraic multiples of each other. Assuming non-CM, the dimension has to be 6. This happens if and only if $J_\l^{new}$ is either a direct sum of two elliptic curves, in which case $End_0(J_\l^{new})$ is a matrix algebra, or $J_\l^{new}$ is a simple abelian variety whose endomorphism algebra is a division algebra. In either case, the endomorphism algebra is a quaternion algebra.

\subsection{A few remarks}
Independent of Theorem \ref{wu}, one can conclude for the cases when $N=4,6$, that if $End_0(J_{\l}^{new})$ contains a quaternion algebra for a generic $\l\in \Q$, then the Beta quotient is algebraic. By the previous discussion,  $\sigma_1$  and $\sigma_{N-1}$, whose traces can be computed by Gaussian hypergeometric functions, always differ by a character,  which yields an algebraic Beta quotient by Yamamoto's result. The useful fact for us is that for general $N$ we can compute $\dim_{\overline\Q}\widehat{V}_{J_\l^{new}}$ for a generic $\l \in \overline \Q$ using Galois representations.

So far, we cannot tell the decomposition of $J_\l^{new}$  using Galois representations. However, the period matrix can be used to determine the endomorphism algebra and the decomposition of $J_\l^{new}$, which will be discussed next.

\section{Periods and endomorphisms}\label{ss:PS}
We now assume that $1\leq i,j,k< N$, $\gcd(N,i,j,k)=1$, $N\nmid i+j$, nor $i+j+k$ and $\l\neq 0, 1$. Let $S$ be a basis of $H^0(X(\l),\Omega^1)^{new}$ collecting differentials of the form $\omega=x^{b_0}(1-x)^{b_1}(1-\l x)^{b_2}dx/y^n$ as in \eqref{eq:1} with $(n,N)=1$ .  In  \cite[Satz 1]{Wolfart}, Wolfart shows that  the Jacobian subvariety $J_\l^{new}$ is isogenious to $\C^{\varphi(N)}/\Lambda(\lambda)$, where
$$
  \Lambda(\l)=\left\{ \(\sigma_n(u)\int_{\gamma_{01}}\omega+\sigma_n(v)\int_{\gamma_{\frac1\l\infty}}\omega\)_{\omega \in S}: u,v \in \Z[\zeta_N]\right\},
$$
and {$\sigma_n$ is the element  of $\mbox{Gal}(\Q(\zeta_N)/\Q)$ such that $\sigma_n(\zeta_N)=\zeta_N^n$}. In particular,
when $N < i+j+k <2N$, the holomorphic differential $1$-forms are given by $\omega_n=dx/y^n$ with $(n,N)=1$.
This means the lattice $ \Lambda(\l)$ can be identified with the $\Z$-module generated by the $2\varphi(N)$ columns
$$
  \(\sigma_n(\zeta_N^i)\int_{\gamma_{01}}\omega_n\)_i, \quad  \(\sigma_n(\zeta_N^i) \int_{\gamma_{\frac1{\lambda}\infty}}\omega_n\)_i,\quad (n,N)=1,
$$
for $i=0, \ldots, \varphi(N)-1$.  For $n=1$, $N-1$, the periods can be computed as before.
A similar computation applies to each $n$ with $(N,n)=1$.  Furthermore, if $i+j\neq N$,  a period matrix of the complex tori $\C^{\varphi(N)}/\Lambda(\lambda)$
can be written as the
$\varphi(N)$-by-$2\varphi(N)$ matrix,
$$
  \(\(\sigma_n(\zeta_N^i)\int_{\gamma_{01}}\omega_n,\)\mid \(\sigma_n(\zeta_N^i) \int_{\gamma_{\frac1{\lambda}\infty}}\omega_n\)\),
$$
where the indices $n$, and $i$ run over all $(N,n)=1$, and $0,\cdots,\varphi(N)-1$, respectively.

\begin{example} For the curve $X_\l^{[6;4,3,1]}$, the corresponding period matrix is given by
$$
\(
\begin{array}{cc|cc}
  \tau_1 &\zeta\tau_1 &\beta_1\tau_3  &\zeta \beta_1\tau_3\\
  \tau_3 &\zeta^{-1}\tau_3 &\beta_2\tau_1  &\zeta^{-1} \beta_2\tau_1
\end{array}
\)
$$
where
\begin{align*}
  \tau_1=  B\(1/3, 1/2\)\hg {\frac16}{\frac13}{\frac56}{\lambda}, & \quad
  \tau_3=  B\( 2/3, 1/2\)\hg{\frac 56}{\frac 23}{\frac76}{\lambda},\\
  \beta_1=  (-1)^{-2/3}\(\lambda^{1/6}(1-\lambda)^{1/3} \sqrt[3]{2}\),& \quad
  \beta_2=  (-1)^{2/3}\(\lambda^{-1/6}(1-\lambda)^{-1/3} \sqrt[3]{4}\),
\end{align*}
and $\beta_1\beta_2=2$.
From this, we can see that the endomorphisms
$$
  \begin{pmatrix}\zeta&0\\0&\zeta^{-1}\end{pmatrix},\quad  \begin{pmatrix}0&\beta_1\\\beta_2&0\end{pmatrix}
$$
are contained in the $End_0(J_\l^{new})$. For a generic choice of $\l\in \Q$ (non-CM case), these two generate $End_0(J_\l^{new})$ which is isomorphic to  $\left ( \frac{-3,2}{\Q}\right)$.

In view of the Gaussian hypergeometric functions, we have for any prime $p\equiv 1 \pmod 6$ and $\eta\in \widehat {\fp^{\times}}$ of order 6
\begin{equation}
  {\hgp{\eta}{\eta^2}{\overline{\eta}}{\lambda}}=\eta(\lambda)\eta^2\(\frac{1-\lambda}4\)\, {\hgp{\overline{\eta}}{\overline{\eta}^2}{{\eta}}{\lambda}},
\end{equation}
where $\eta$ is a character in $\widehat{\F_p^\times}$ of order $6$ with $\eta(0)=0$.  This equality follows by Proposition \ref{prop:9} and \cite[Chapter 3]{Berndt-Evans-Williams}.
\end{example}

In general, when $N=3, 4, 6$, a period matrix of $\C^2/\Lambda$  {corresponding to $J_\l^{new}$} is given by
$$
\(
\begin{array}{cc|cc}
  \tau_1 &\zeta_N\tau_1 &\alpha(\l)\beta\tau_2 &\zeta_N \alpha(\l)\beta\tau_2\\
  \tau_2 &\zeta_N^{-1}\tau_2 &\gamma\tau_1/\beta\alpha(\l) &\zeta_N^{-1} \gamma\tau_1/\beta\alpha(\l)
\end{array}
\),
$$
where
$$
  \tau_1= \int_{\gamma_{01}}\omega_1, \quad \tau_2= \int_{\gamma_{01}}\omega_{N-1},
$$
$$
  \beta=B\(\frac{i+j+k-N}N,\frac{N-k}{N}\)/B\(\frac iN,\frac jN\)
$$
and
$\alpha(\l)$, $\gamma$ are the same notations in equation (\ref{eq:18}).
Thus,  if $\beta\in\overline\Q$ then $\mbox{End}(J_\l^{new})$ contains the endomorphisms
$$
  E=\begin{pmatrix}\zeta_N&0\\0&\zeta_N^{-1}\end{pmatrix},\quad  J=\begin{pmatrix}0&\alpha(\l)\beta\\ \frac{\gamma}{\alpha(\l)\beta}&0\end{pmatrix}.
$$
Furthermore, it contains the quaternion algebra defined over $\Q$ generated by $I=2E-(\zeta_N+\zeta_N^{-1})$  and $J$
with
$$
  I^2= \(\zeta_N-\zeta_N^{-1}\)^2,  \quad J^2=\gamma \quad   \in {\Q(\zeta_N+\zeta_N^{-1})},
$$ the field $\Q(\zeta_N+\zeta_N^{-1})$ is $\Q$ when $N=3,4,6$.

\section{2-dimensional abelian varieties related to (3,6,6)}
Now we compare our result with a discussion in \cite{PS} for (3,6,6).  {We note our notation differs from \cite{PS} by changing $s^2$ to $s$.} Consider the family of Picard curves
$$
  C'(s):\, w^3=z(z-1)\((z-1/2)^2-s/4\),
$$
which are isomorphic to
$$
   C(s):\, w^3=(z^2-1/4)\(z^2-s/4\).
$$
The Jacobians of these genus 3 Picard curves decompose into a direct sum of a CM elliptic curve $E'(s)$ and a $2$-dimensional abelian variety $A'(s)$ such that for each ${s}\in \overline \Q\setminus \{0,1\}$, $End_0(A'(s))$ contains $\left ( \frac{-3,2}{\Q}\right)$ (see \cite{PS}).  In other words, one can take $A'(s)$, up to isogeny, to be the family of genus 2 curves with QM by $\left (\frac{-3,2}{\Q}\right )$ above the Shimura curve for $(3,6,6)$.

In \cite{PS}, it is shown that the holomorphic differential $\frac{dz}{w}$ on $C'(s)$  satisfies an order-2 Picard-Fuchs equation, with one solution
$
\pFq{2}{1}{\frac 16, \frac 13}{, \frac 56}{s}.
$
\footnote{In \cite{PS} the function is listed as $\pFq{2}{1}{\frac 16, \frac 13}{\quad \frac 23}{s}$, which was a minor error.}

Suppose
$$
   \mbox{Jac}(C(s))=E^\prime(s)\oplus A^\prime(s)
$$
and
$$
  \mbox{Jac}(X_s^{[6;4,3,1]})=E(s)\oplus A(s),
$$
where for $\mu=e^{2\pi i/3}$,
$$
  E^\prime(s):\, w^3=(z-1/4)\(z-s/4\) \simeq \C\diagup (\Z\mu+\Z)
$$
and
$$
  E(s):\, y^3=x^4(1-x)^3(1-sx)
$$
are both CM elliptic curves, but not isomorphic over $\Q(s)$ in general. We then claim that $A(s)$ is $\Q$-isogenous to $A^\prime(s)$ if $s\in \Q\setminus \{ 0, 1\}$.
\begin{theorem}
Let $s\in \Q$, $\ell$ be prime, and $\rho$, $\rho'$ the 4-dimensional  $\ell$-adic Galois representations of $G_\Q$ arising from  $A(s)$ and $A'(s)$, respectively. If both $\rho$ and $\rho'$ are absolutely irreducible, then they are isomorphic.
\end{theorem}
\begin{proof}
If both $\rho$ and $\rho'$ are absolutely irreducible, it suffices to check that they have the same trace at  Frobenius $\text{Frob}_p$  for almost all unramified primes $p>2$.  When $p\equiv 2 \pmod 3$, we see that the traces of $\text{Frob}_p$ are both zero by the way the algebraic curves $ y^3=(x^2-1/4)\(x^2-s/4\)$, and $y^6=x^4(1-x)^3(1-sx)$ are presented, and using the fact that $y\mapsto y^3$ is a bijection on $\mathbb F_p$.

Now we assume $p\equiv 1\pmod 3$.  From the same point-counting technique that we used in \S \ref{counting}, it is sufficient to show that for any  $\eta\in \widehat{\mathbb F_p^\times}$ of order 6,
\begin{multline}\label{eqn:366target}
 \sum_{x\in \fp}\left[\eta^2(f(x^2))+\eta^4(f(x^2))-\eta^2(f(x)) -\eta^4(f(x)) \right]\\
= \sum_{x\in \fp}  \left[\eta(g(x))+ \eta^5(g(x)) \right],
\end{multline}
where $f(x)=(x-1/4)\(x-s/4\)$ and $g(x)=x^4(1-x)^3(1-sx)$.

Let $H=Gal(\overline{\Q}/\Q(\sqrt{-3}))$. Both $\rho$ and $\rho'$ are induced from 2-dimensional Galois representations of $H$, i.e. $\rho= \text{Ind}_H^G \sigma$ and $\rho'= \text{Ind}_H^G \sigma',$ for some 2-dimensional Galois representations $\sigma$ and $\sigma'$ of $H$. It suffices to check that upon choosing the right $\sigma$ and $\sigma'$ among conjugates, $\Tr \sigma(\text{Frob}_p)=\Tr \sigma'(\text{Frob}_p)$. In our setting, the claim to be established is

\begin{equation}
 \sum_{x\in \F_p} \left[\eta^2(f(x^2))-\eta^2(f(x))\right]
= \sum_{x\in \F_p}  \eta(g(x)).
\end{equation}

Both  $\Tr \rho(\text{Frob}_p)$ and  $\Tr \rho'(\text{Frob}_p)$ are rational integers with absolute values less than ${4 \cdot \sqrt{p}}$ using the Weil bound. When $p>16$, it suffices to check that
\begin{equation*}
\Tr \rho(\text{Frob}_p)\equiv \Tr \rho'(\text{Frob}_p) \pmod p.
\end{equation*}

From this argument, \eqref{eqn:366target} will be a consequence of
\begin{equation}\label{simpler}
 \sum_{x\in \F_p} \eta^2(f(x^2))-\eta^2(f(x))
\equiv \sum_{x\in \F_p}  \eta(g(x)) \pmod p.
\end{equation}

To verify \eqref{simpler} modulo $p$, it suffices to observe that the constants of $((x-1/4)(x-s/4))^{(p-1)/3}$ and $((x^2-1/4)(x^2-s/4))^{(p-1)/3}$ agree, and then prove that the $(p-1)$st coefficients of  $((x^2-1/4)(x^2-s/4))^{(p-1)/3}$ and $(x^4(1-x)^3(1-sx))^{(p-1)/6}$ agree modulo $p$, which is proved in Lemma \ref{p-1coeffs} below.
\end{proof}

\begin{lemma}\label{p-1coeffs}
{Let $p\equiv 1 \pmod{3}$ be prime.}  The $(p-1)$st coefficients of  $((x^2-1/4)(x^2-s/4))^{(p-1)/3}$ and $(x^4(1-x)^3(1-sx))^{(p-1)/6}$ agree modulo $p$.
\end{lemma}

\begin{proof}
Let $f(x),g(x)$ be defined as in $(\ref{eqn:366target})$, and {consider the integer $m=(p-1)/6$.}  Using the binomial theorem, we see that
\begin{align*}
f(x^2)^{2m} & = \left[ \sum_{i=0}^{2m} \CC{2m}{i}x^{2(2m-i)}\left(\frac{-1}{4} \right)^i \right] \cdot \left[ \sum_{j=0}^{2m} \CC{2m}{j} x^{2(2m-j)} s^j \left(\frac{-1}{4}\right)^j \right] \\
& = \sum_{i,j=0}^{2m} \CC{2m}{i}\CC{2m}{j}\left(\frac{-1}{4}\right)^{i+j}s^jx^{8m-2(i+j)}
\end{align*}
and
\begin{align*}
g(x)^m & = x^{4m}\cdot \left[ \sum_{i=0}^{3m} \CC{3m}{i} (-1)^{3m-i} x^{3m-i} \right] \cdot \left[ \sum_{j=0}^m \CC{m}{j} (-1)^{j} s^j x^j \right] \\
& = \sum_{i=0}^{3m} \sum_{j=0}^m \CC{3m}{i} \CC{m}{j} (-1)^{3m-i+j} s^j x^{7m-i+j}.
\end{align*}
The $(p-1)$st coefficient of $f(x^2)^{2m}$ { occurs when $i=m-j$, and is thus}
\begin{align*}
&\sum_{j=0}^{m}\CC{2m}{m-j}\CC{2m}j\(\frac{-1}4\)^{m}s^{j} = \sum_{j=0}^{m}\frac{(2m)!^2}{(m-j)!(m+j)!j!(2m-j)!}\(\frac{-1}4\)^{m}s^{j},
\end{align*}
and the $(p-1)$st coefficient of $g(x)^m$ { occurs when $i=m+j$, and is thus}
\[
\sum_{j=0}^{m}\CC{3m}{m+j}\CC{m}js^{j} = \sum_{j=0}^{m}\frac{(3m)!m!}{(m+j)!(2m-j)!j!(m-j)!}s^{j}.
\]

Thus it suffices to show that $(3m)!m! \equiv (2m)!^2\left(\frac{-1}{4}\right)^m \pmod{p}$.  Equivalently, we claim that
\[
\frac{2^m(3m)!}{(2m)!} \equiv \frac{(-1)^m (2m)!}{2^m m!}\pmod p.
\]
To verify this, observe that since $6m=p-1$,
\begin{align*}
2^m(3m)!/(2m)! &= (p-1)(p-3) \cdots (p-(2m-1)) \\
& \equiv (-1)^m (2m-1)(2m-3)\cdots(2m-(2m-1)) \pmod p \\
& \equiv \frac{(-1)^m (2m)!}{2^m m!}\pmod p.
\end{align*}
\end{proof}


\section{Other cases}
\subsection{ $X_\l^{[3;1,2,1]}$}

\begin{theorem}Let  $\l\in \Q\setminus\{0,1\}$ and $\rho$ be   the
4-dimensional Galois representation of $G_\Q$ arising from the
genus-2 curve $y^3=x(x-1)^2(1-\l x)$. Let $\rho'$ be the Galois
representation of $G_\Q$ arising from the elliptic curve
$y^2+xy+\frac{\l}{27}=x^3$. For any $\l\in \Q$ such that the elliptic
curve does not have complex multiplication,
$\rho$ is isomorphic to $\rho'\oplus ( \rho'\otimes\chi_{-3})$ where
$\chi_{-3}$ is the quadratic character of $G_\Q$ with kernel
$G_{\Q(\sqrt{-3})}$.
\end{theorem}

\begin{proof} By the automorphism $(x,y)\mapsto(x,\zeta_3^{-1} y)$,
$\rho|_{G_\Q(\sqrt{-3})}\cong \sigma_1\oplus \sigma_2$. By {Corollary}
\ref{cor:8},  $\sigma_1\cong \sigma_2$. Thus  each representation can
be lifted to a 2-dimensional Galois representation $\pi$ of $G_\Q$.
By Theorem 1.1 of Lennon \cite{Lennon}, each $\sigma$ is
isomorphic to  $\rho'|_{G_\Q(\sqrt{-3})}$. Thus $\pi$ and $\rho'$
differ by at most a character of $G_\Q$ with kernel containing
$G_{\Q(\sqrt{-3})}$.  The last claim follows from $\Tr
\rho(\Frob_q)=0$ when $q\equiv 2 \pmod 3$.
\end{proof}
The relation between the periods are
\begin{align*}
  \tau_1=  B\(1/3, 2/3\)\hg {\frac13}{\frac23}{1}{\lambda},
  \tau_1^\prime= -\tau_1,\quad
  \tau_2=\tau_1 ,\quad
  \tau_2^\prime=&-\tau_1.
\end{align*}
For generic choice of $\l\in \Q$, $End_0(J_\l^{new})$ is a matrix algebra as $J_\l^{new}$ is isogenous to two elliptic curves.

From the Galois representation perspective, this case is analogous to the discussion of 4-dimensional Galois
representation admitting QM over a quadratic field in \cite{ALLL}.  A special case which arises from a 4-dimensional Galois representations
attached to noncongruence modular forms was discussed in detail in \cite{LLY}. The family of elliptic curves $y^2+xy+\frac{\l}{27}=x^3$ also shows up in \cite{CDLNS}  on $p$-adic analogues of Ramanujan-type formulas for $1/\pi$.

\subsection{$X_\l^{[12;9,5,1]}$}
The arithmetic group $\Gamma = (2,6,6)$ can be realized as the monodromy group of a period on $J_\l^{[12;9,5,1]}$. The quaternion algebra $H_\Gamma$ corresponding to (2,6,6) is the indefinite quaternion algebra defined over $\Q$ with discriminant $6$. For the subvariety $J_\l^{{new}}$, the lattice $\Lambda(\l)$ is generated by
$$
\begin{array}{cccc|cccc}
  \tau_1, &\zeta\tau_1, &\zeta^2\tau_1, &i\tau_1,& i\l^{\frac 16}\alpha\tau_3, &\zeta i\l^{\frac 16}\alpha\tau_3, &\zeta^2i\l^{\frac 16}\alpha\tau_3, &-\l^{\frac 16}\alpha\tau_3\\
  \tau_3, &\tau_3/\zeta, &\tau_3/\zeta^2, &-i\tau_3,& i\frac{2+\sqrt 3}{\alpha\l^{\frac 16}}\tau_1, & i\frac{2+\sqrt 3}{\alpha\zeta\l^{\frac 16}}\tau_1, &i\frac{2+\sqrt 3}{\alpha\zeta^2\l^{\frac 16}}\tau_1, &\frac{2+\sqrt 3}{\alpha\l^{\frac 16}}\tau_1\\
  \alpha\tau_3, &\zeta^5\alpha\tau_3, &\alpha\tau_3/\zeta^2, &i\alpha\tau_3,& i\tau_1/\l^{\frac 16}, &\zeta^5 i\tau_1/\l^{\frac 16}, &i\tau_1/\zeta^2\l^{\frac 16}, &-\tau_1/\l^{\frac 16}\\
  \frac{2+\sqrt 3}{\alpha}\tau_1, &\frac{2+\sqrt 3}{\alpha\zeta^5}\tau_1, &\frac{2+\sqrt 3}{\alpha\zeta^{-2}}\tau_1, &\frac{2+\sqrt 3}{i\alpha}\tau_1,& i\l^{\frac 16}\tau_3, & i\l^{\frac 16}\tau_3/\zeta^5, &\zeta^2i\l^{\frac 16}\tau_3, &\l^{\frac 16}\tau_3
  \end{array}
$$
where
$$
  \tau_1=  B\(1/4, 7/12\)\hg{\frac1{12}}{\frac14}{\frac56}{\lambda},
  \tau_3=  B\( 5/12, 3/4\)\hg{\frac 34}{ \frac {11}{12}}{\frac76}{\lambda},
$$
$$  \alpha= (1-\lambda)^{1/2} \sqrt{9+6\sqrt 3}/3.$$
Generically, $\dim_{\overline \Q} \hat V_{J_\l^{new}}=6$. In particular, we can see that $\mbox{End}(J_\l^{{new}})$ is generated by the endomorphisms
$$
A=\begin{pmatrix}\zeta&0&0&0\\0&1/\zeta&0&0\\0&0&\zeta^5&0
\\0&0&0&1/\zeta^5\end{pmatrix}, \quad
  B=\begin{pmatrix}0&0&i/\l^{\frac 16}&0\\0&0&0&i\l^{\frac 16}\\i\l^{\frac 16}&0&0&0\\0&i/\l^{\frac 16}&0&0\end{pmatrix},$$
$$  C=\begin{pmatrix}0&i\frac{2+\sqrt 3}{\alpha\l^{\frac 16}}&0&0\\i\l^{\frac 16}\alpha&0&0&0\\0&0&0&\frac{i\l^{\frac 16}}{\alpha}\\0&0&i\frac{\alpha\l^{-\frac 16}}{2+\sqrt3}&0\end{pmatrix},
$$
with the relations
$$
  A^4-A^2=-1,\quad B^2=-1,\quad C^2+A+A^{-1}=-2,
$$
and
$$
  BAB^{-1}=A^3-A, \quad CAC^{-1}=A^{-1}, \quad CBC^{-1}=(2+A+A^{-1})B.
$$
Moreover, $End_0(J_\l^{{new}})$ contains the quaternion algebra $\(\frac{-1,3}\Q\)\simeq H_\Gamma$, which is generated by $B$, and $A+A^{-1}$.

In this case {$\int_0^1\omega_1/\int_{\frac 1\l}^\infty\omega_{11}$} is algebraic as $B(1/4,7/12)/B(1/12,3/4)=\sqrt{\frac{2\sqrt{3}}3 - 1}.$ For the Gaussian hypergeometric functions, if $\lambda \neq 0, 1 \in \Bbb F_p$, we have the identities:
\begin{align*}
\hgp{\eta}{\eta^3}{\eta^{-2}}{\lambda} & = \eta^2(\lambda)\hgp{\eta^5}{\eta^3}{\eta^2}{\lambda}\\
& =\eta\(-27(1-\lambda)^6\) \, \hgp{\eta^{-5}}{\eta^{-3}}{\eta^{-2}}{\lambda}\\
& =\eta\(-27\lambda^2(1-\lambda)^6 \) \, \hgp{\eta^{-1}}{\eta^{-3}}{\eta^{2}}{\lambda},
\end{align*}
where $\eta$ is a multiplicative character of $\fpc$ of order $12$.

\subsection{$X_{\l}^{[10;2,7,7]}$} For the abelian variety
$J_\l:=J_{\l}^{[10;2,7,7]}$, the corresponding periods of $J_\l^{new}$
are
\begin{align*}
  \tau_1=&\int_0^1\omega_1= B\(3/10,
4/5\)\hg{\frac7{10}}{\frac45}{\frac{11}{10}}{\lambda}, \\
  \tau_2=&\int_0^1\omega_9= B\(7/10,
1/5\)\hg{\frac3{10}}{\frac15}{\frac{9}{10}}{\lambda},  \\
  \tau_3=&\int_0^1\omega_3= B\(9/10,
2/5\)\hg{\frac1{10}}{\frac25}{\frac{13}{10}}{\lambda}, \\
  \tau_4=&\int_0^1\omega_7= B\(1/10,
3/5\)\hg{\frac9{10}}{\frac35}{\frac{7}{10}}{\lambda},
\end{align*}
and
\begin{align*}
  \tau_1'=&\int_1^\infty\omega_1= \frac{\sqrt
5-1}{2\alpha_1(\l)\beta_1}\tau_2,\,
  \tau_2'=\int_1^\infty\omega_9=\alpha_1(\l)\beta_1\tau_1 \\
  \tau_3'=&\int_1^\infty\omega_3=\frac{-\sqrt
5-1}{2\alpha_1(\l)\beta_2}\tau_4,\,
  \tau_4'=\int_1^\infty\omega_7=  \alpha_2(\l)\beta_2\tau_3 \\
\end{align*}
where
\begin{align*}
  \alpha_1(\l)&=(-1)^{7/5}\l^{1/10}(1-\l)^{2/5},\, \beta_1=B\(7/10,
2/5\)/B\(3/10, 4/5\),\\
  \alpha_2(\l)&=(-1)^{1/5}\l^{3/10}(1-\l)^{-4/5},\, \beta_2=B\(1/10,
1/5\)/B\(9/10, 2/5\).
\end{align*}
The lattice $\Lambda(\l)$ is generated by
$$
\begin{array}{cccc|cccc}
  \tau_1, &\zeta\tau_1, &\zeta^2\tau_1, &\zeta^3\tau_1,&  \frac{\sqrt
5-1}{2\alpha_1(\l)\beta_1}\tau_2, &\zeta  \frac{\sqrt
5-1}{2\alpha_1(\l)\beta_1}\tau_2, &\zeta^2 \frac{\sqrt
5-1}{2\alpha_1(\l)\beta_1}\tau_2, & \zeta^3\frac{\sqrt
5-1}{2\alpha_1(\l)\beta_1}\tau_2\\
  \tau_2, &\frac{\tau_2}\zeta, &\frac{\tau_2}{\zeta^2},
&\frac{\tau_2}{\zeta^3},&\alpha_1(\l)\beta_1\tau_1 , &
\frac{\alpha_1(\l)\beta_1\tau_1}\zeta,
&\frac{\alpha_1(\l)\beta_1\tau_1}{\zeta^2},
&\frac{\alpha_1(\l)\beta_1\tau_1 }{\zeta^3}\\
  \tau_3, &\zeta^3\tau_3, &\zeta^6{\tau_3}, &\frac{\tau_3}\zeta,&
\frac{-\sqrt 5-1}{2\alpha_2(\l)\beta_2}\tau_4, &\zeta^3 \frac{-\sqrt
5-1}{2\alpha_2(\l)\beta_2}\tau_4, & \frac{-\sqrt
5-1}{2\alpha_2(\l)\beta_2\zeta^4}\tau_4, & \frac{-\sqrt
5-1}{2\alpha_2(\l)\beta_2\zeta}\tau_4\\
 \tau_4, &\frac{\tau_4}{\zeta^3}, &\frac{\tau_4}{\zeta^6},
&\zeta\tau_4,&\alpha_2(\l)\beta_2\tau_3,
&\frac{\alpha_2(\l)\beta_2\tau_3}{\zeta^3}, &
\frac{\alpha_2(\l)\beta_2\tau_3}{\zeta^6},
&\zeta\alpha_2(\l)\beta_2\tau_3.
  \end{array}
$$
By using Gaussian hypergeometric functions, one knows that the
subrepresentations  $\sigma_{m}$ and $\sigma_{N-m}$ differ by a character, as in \S \ref{ss:Galois 3.5}. Thus
$\beta_1$, $\beta_2$ are both algebraic. Meanwhile, $\sigma_1$ and
$\sigma_3$ do not differ
by a character.  Thus, combining with W\"ustholz's result we know that
for a generic $\l\in \overline \Q$, the 4-dimensional abelian variety
$J_\l^{new}$ is simple, and $\Lambda_{\overline\Q}(J_\l^{new})$ is
10-dimensional.

We can see that ${End_0}(J_\l^{new})$ contains the endomorphisms
$$
  A=\begin{pmatrix}\zeta&0&0&0\\0&\zeta^{-1}&0&0\\0&0&\zeta^3&0\\0&0&0&\zeta^{-3}\end{pmatrix},\quad
  B=\begin{pmatrix}0&\alpha_1(\l)\beta_1&0&0\\\frac{\sqrt
5-1}{2\alpha_1(\l)\beta_1}&0&0&0\\0&0&0&\alpha_2(\l)\beta_2\\0&0&
\frac{-\sqrt 5-1}{2\alpha_2(\l)\beta_2}&0\end{pmatrix},
$$
which satisfy the relations
$$
  A^4-A^3+A^2-A=-1,\quad B^2=A^2+A^{-2},\quad BAB^{-1}=A^{-1}.
$$
{The algebra $End_0(J_\l^{{new}})$ contains the quaternion algebra {
\[
\(\frac{\frac{\sqrt 5-5}2,\frac{\sqrt 5-1}2}{\Q(\sqrt{5})}\),
\]
which is 
the quaternion algebra defined over $\Q(\sqrt5)$ with discriminant $\mathfrak{p}_5$.}

\subsection{$X_{2}^{[5;1,4,1]}$}
For the curve $y^5=x(1-x)^4(1-2x)$, from Galois perspective, its
L-function is expected to be related to two Hilbert modular forms, which differ by
embeddings of $\Q(\sqrt 5)$ to $\mathbb C$. From numeric data, we
identified two Hilbert modular forms, which are labeled by Hilbert
Cusp Form 2.2.5.1-500.1-a in the LMFDB online database \cite{lmfdb}.

Here, we tabulate the local $L$-functions of the curve $C:\,
y^5=x(1-x)^4(1-2x)$ with factorization over $\Q(\sqrt{5})$, and the
trace data for the Hilbert cusp forms for some small primes.
$$
  \begin{array}{c|c|c}
     p & L_p(C,T) \mbox{ over } \Q(\sqrt{5})& \mbox{Hecke eigenvalues
}\\ \hline\hline
     7 & (49T^4 + 10T^2 + 1)(49T^4 - 10T^2 + 1)&-10\\ \hline
     11& (11T^2 - 2T + 1)^4 & 2,2\\ \hline
     13&  (169T^4 + 1)^2&0\\ \hline
     17& (289T^4 - 20T^2 + 1)(289T^4 +20T^2 + 1)&20\\ \hline
     19& \begin{array}{c}
           \(19T^2-5\(\frac{1+\sqrt 5}2\)T+1\)\(19T^2-5\(\frac{1-\sqrt
5}2\)T+1\)\\
           \(19T^2+5\(\frac{1+\sqrt 5}2\)T+1\)\(19T^2+5\(\frac{1-\sqrt
5}2\)T+1\)
          \end{array}& 5\(\frac{1\pm\sqrt 5}2\)\\ \hline

     31& \(\(31T^2+\(\frac{1+5\sqrt 5}2\)T+1\)\(31T^2+\(\frac{1-5\sqrt
5}2\)T+1\)\)^2&\frac{-1\pm5\sqrt 5}2  \\ \hline
     41&\(\(41T^2+\(\frac{1+5\sqrt 5}2\)T+1\)\(41T^2+\(\frac{1-5\sqrt
5}2\)T+1\)\)^2&\frac{-1\pm5\sqrt 5}2
  \end{array}
$$

\section{Acknowledgements}
The project was initiated at the Women in Numbers 3 (WIN3) workshop held at Banff International Research Station in April 2014. We are grateful for the opportunity to collaborate. We would like to thank the National Center for Theoretical Sciences (NCTS) in Taiwan for supporting Fang-Ting Tu to attend WIN3 and visit Ling Long.  Long and Tu are supported in part by NSF DMS1303292 and DMS1001332.  The fourth author would like to thank Tulane University for hosting her while working on this project.  She also thanks her Women in Sage 5 group Heidi Goodson, Anna Haensch, and Alicia Marino for help with code to compute classes of examples of this work.  We are in debt also for enlightening discussions with Noam Elkies, Jerome W. Hoffman, Kumar Murty, Matt Papanikolas, and Paula Tretkoff.   We also are grateful for the L-functions and Modular Forms Database \cite{lmfdb}.  Many computations were done using \texttt{sage, magma, maple}.



\end{document}